\documentclass[reqno,12pt]{amsart}
\usepackage{amssymb,amsthm,amsmath,amsfonts,amstext,mathrsfs,url}
\usepackage{latexsym}
\usepackage{fancyhdr}
\usepackage{helvet}

\newtheorem{thm}{Theorem}

\newtheorem*{prob}{Problem}
\newtheorem{prop}{Proposition}
\newtheorem{defin}{Definition}

\input xy
\xyoption{all}

\newcommand{\m}{\mathbf}

\def\d{\,{\rm{d}}}
\def\v{\,{\varsigma}}
\def\ll{\mathbf{l}}

\title[Commutative rational vector fields]
{Planar $2$-homogeneous\\commutative rational vector fields}
\author[G. Alkauskas]{Giedrius Alkauskas}
\address{Vilnius University, Institute of Computer Science, Naugarduko 24, LT-03225 Vilnius, Lithuania}
\email{giedrius.alkauskas@mif.vu.lt}

\begin{document}
\begin{abstract} In this paper we prove the following result: if two $2$-dimensional $2$-homogeneous rational vector fields commute, then either both vector fields can be explicitly integrated to produce rational flows with orbits being lines through the origin, or both flows can be explicitly integrated in terms of algebraic functions. In the latter case, orbits of each flow are given in terms of $1$-homogeneous rational functions $\mathscr{W}$ as curves $\mathscr{W}(x,y)=\textrm{const}$. An exhaustive method to construct such commuting algebraic flows is presented. The degree of the so-obtained algebraic functions in two variables can be arbitrarily high. 
\end{abstract}

\pagestyle{fancy}
\fancyhead{}
\fancyhead[LE]{{\sc Commutative projective flows}}
\fancyhead[RO]{{\sc G. Alkauskas}}
\fancyhead[CE,CO]{\thepage}
\fancyfoot{}

\date{\today}
\thanks{The research of the author was supported by the Research Council of Lithuania grant No. MIP-072/2015}

\maketitle

\section{Introduction} In this paper, by ``smooth", we mean of class $C^{\infty}$ in $\mathbb{R}^{n}$, in domains where rational functions in consideration are defined. Since we are mainly interested in rational vector fields, these will produce (possibly, ramified) flows in all $\mathbb{R}^{n}$. If a plane vector field $\varpi\frac{\partial}{\partial x}+\varrho\frac{\partial}{\partial y}$ is not used as a derivation on the space of algebra of $C^{\infty}$ germs and is given in cartesian coordinates (the second condition will always be satisfied), we will write a vector field as $(\varpi,\varrho)$.

\subsection{Motivation}
\label{sub-motiv}
 Let $F(\m{x},t):\mathbb{R}^{n}\times\mathbb{R}\mapsto\mathbb{R}^{n}$ be a flow with any smooth vector field. Then it satisfies the \emph{translation equation}, given by \cite{aczel, conlon, gadea, nikolaev}
\begin{eqnarray*}
F(F(\m{x},z),w)=F(\m{x},z+w),\quad \m{x}\in\mathbb{R}^{n},\quad z,w\in\mathbb{R}\text{ are small enough}.
\end{eqnarray*} 
One should make a distinction between \emph{local} and \emph{global} flows \cite{conlon}.\\
 
 In books on differential geometry, calculus, differential equations and vector fields (like \cite{conlon, gadea, krasno, nikolaev}) one usually considers various vector fields in $\mathbb{R}^{n}$, given, for example, by holomorphic or meromorphic functions. A special case of interest, with a strong algebraic and geometric emphasis, is to consider polynomial vector fields, or even homogeneous polynomial vector fields; for example, \cite{camacho}. A closely related topic is that of \emph{algebraic differential equations}.\\
 
  Let us now not limit to the polynomial case, but rather investigate rational functions as coordinates of vector fields. This introduces methods from birational geometry into the subject - for example, Cremona groups \cite{hudson}. This has a strong algebro-geometric flavour. Indeed, if $\ell$ is a birational transformation of $\mathbb{R}^{n}$, then the function
\begin{eqnarray*}
F^{\ell}(\m{x},t)=\ell^{-1}\circ F(\ell(\m{x}),t)
\end{eqnarray*} 
is also a flow with a rational vector field, which can be directly calculated from $\ell$ and the original vector field. \\

Suppose now, we wish to investigate which $n$-dimensional rational vector fields produce flows $F$, which are:
\begin{itemize}
\item[i)]rational;
\item[ii)]algebraic;
\item[iii)]flows with exactly $i$, $1\leq i\leq n-2$, independent rational first integrals;
\item[iv)]flows whose orbits are algebraic curves ($n-1$ independent rational first integrals - see \cite{chav} for planar polynomial case, and \cite{schlomuk} for the exposition in planar case also);
\item[v)]unramified flows (see \cite{alkauskas-un},  and also (\cite{alkauskas-super1}, Section 2.1) for the precise definition of the term \emph{unramified flow}).
\end{itemize}
Then exactly the same property is shared by $F^{\ell}$. Thus, we may wish to classify such flows up to birational equivalence. This is a new way of looking at flows with a rational vector field which have an intrinsic arithmetic structure, as itemized above.\\

We could worry also, for example, not about rational first integrals, but elementary first integrals, or \emph{Liouvillian} ones, which are functions that are built up from rational functions using exponentiation, integration, and algebraic functions. For $2$-dimensional polynomial vector fields, these questions are investigated in \cite{christopher,prelle,singer}.\\

One can also consider $1$-dimensional complex case and rational vector fields. For example, a differential equation $\dot{z}=R(z)$ is studied in \cite{benzinger}.

\subsection{$2$-homogeneous case}Now, instead of looking at general rational vector fields, in a series of papers \cite{alkauskas,alkauskas-2,alkauskas-un,alkauskas-ab}, and also in  \cite{alkauskas-super1,alkauskas-super2,alkauskas-super3, alkauskas-super4} (though more general vector fields are dealt with in the last four papers) we embarked on the task to develop the above program in a special case when a vector field is given by a collection of $2$-homogeneous rational functions.\\

 Indeed, suppose the flow which integrates a fixed $n$-dimensional rational vector field, is given by the function $F(\m{x},z)$, $\m{x}\in\mathbb{R}^{n}$, $z\in\mathbb{R}$. But now, if the vector field is $2$-homogeneous, the time variable can be accommodated within space variables, in a sense that there exists a function $\phi:\mathbb{R}^{n}\mapsto\mathbb{R}^{n}$, such that \cite{alkauskas}
\begin{eqnarray*}
F(\m{x},z)=\frac{\phi(\m{x}z)}{z},\quad \m{x}\in\mathbb{R}^{n},\quad z\in\mathbb{R}.
\end{eqnarray*}

In this case, instead of the translation equation, we have \emph{the projective translation equation}, which was first introduced in \cite{alkauskas-t} and which is the equation of the form
\begin{eqnarray}
\frac{1}{z+w}\,\phi\big{(}\m{x}(z+w)\big{)}=\frac{1}{w}\,\phi\Big{(}\phi(\m{x}z)\frac{w}{z}\Big{)},\quad w,z\in\mathbb{R}.
\label{funk}
\end{eqnarray}
Generally speaking, one can even confine to the case $w=1-z$ without altering the set of solutions, though some complications arise concerning ramifications (when dealing with global flows). In a $2$-dimensional case, $\phi(x,y)=(u(x,y),v(x,y))$ is a pair of functions in two real (or complex) variables. A non-singular solution of this equation is called \emph{a projective flow}. Let $\phi^{z}(\m{x})=z^{-1}\phi(\m{x}z)$. The \emph{non-singularity} means that a flow satisfies the boundary condition
\begin{eqnarray}
\lim\limits_{z\rightarrow 0}\phi^{z}(\m{x})=\m{x}.
\label{init}
\end{eqnarray}

Then a \emph{vector field} is given by
\begin{eqnarray}
\big{(}\varpi(x,y),\varrho(x,y)\big{)}=\frac{\d}{\d z}\frac{\phi(xz,yz)}{z}\bigg{|}_{z=0},
\label{vec}
\end{eqnarray}
which, for projective flows, is necessarily a pair of $2$-homogeneous functions. For smooth functions, the functional equation (\ref{funk}) and the non-singularity condition imply a linear first order PDE \cite{alkauskas}
\begin{eqnarray}
u_{x}(x,y)(\varpi(x,y)-x)+
u_{y}(x,y)(\varrho(x,y)-y)=-u(x,y),\label{pde}
\end{eqnarray}
and the same PDE for $v$, with boundary conditions as given by (\ref{init}):
\begin{eqnarray}
\lim\limits_{z\rightarrow 0}\frac{u(xz,yz)}{z}=x,\quad
\lim\limits_{z\rightarrow 0}\frac{v(xz,yz)}{z}=y.
\label{bound}
\end{eqnarray}
These two PDEs (\ref{pde}) with the above boundary conditions are equivalent to (\ref{funk}) for $z,w$ small enough \cite{alkauskas}.\\

 Alternatively, the whole flow can be described in terms of the system of autonomous ODEs
\begin{eqnarray*}
\left\{\begin{array}{l}
x'(t)=\varpi\big{(}x(t),y(t)\big{)},\\
y'(t)=\varrho\big{(}x(t),y(t)\big{)}.
\end{array}
\right.
\end{eqnarray*}

Each point under a flow $\phi$ possesses the orbit, which is defined by
\begin{eqnarray*}
\mathscr{V}(\m{x})=\bigcup\limits_{z\in\mathbb{R}}\phi^{z}(\m{x})=F(\m{x},\mathbb{R}).
\end{eqnarray*}

The orbits of the flow with the vector field $(\varpi,\varrho)$ are given by $\mathscr{W}(x,y)=\mathrm{const}.$, where the function $\mathscr{W}$ can be found from the differential equation
\begin{eqnarray}
\mathscr{W}(x,y)\varrho(x,y)+\mathscr{W}_{x}(x,y)[y\varpi(x,y)-x\varrho(x,y)]=0,
\label{orbits}
\end{eqnarray}
and $\mathscr{W}$ is uniquely defined from this ODE and the condition that it is a $1$-homogeneous function. \\

Now, recall the next two definitions from \cite{alkauskas-ab}: 
\begin{defin}If there exists a positive integer $N$ such that $\mathscr{W}^{N}(x,y)$ is a rational function (then necessarily $N$-homogeneous), such a smallest positive $N$ is called \emph{the level of the flow}, and the flow itself is called \emph{an abelian flow of level} N. 
\end{defin}
The exception is when $x\varrho-y\varpi=0$ (since then the ODE (\ref{orbits}) is void), in which case a flow is called \emph{level $0$ flow}. For rational vector fields, such a flow of level $0$ is rational, too.
\begin{defin}We call the flow $\phi$ \emph{an algebraic projective flow}, if its vector field is rational, and $(u,v)$ is a pair of algebraic functions.
\end{defin}
Algebraic flow is necessarily an abelian flow of a certain level $N$.\\

We note that, among numerous motivations to investigate $2$-homogeneous vector fields separately, one is completely apt in the setting of the current paper; see Theorem \ref{thm-main}. Namely, that the property of the flow being \emph{algebraic} is unambiguous only if a vector field is $2$-homogeneous. See (Note 3 in \cite{alkauskas-super1}) for an example where it is shown that if a vector field is not such, then there are two possibilities what to call \emph{an algebraic flow}.
\subsection{Commutativity} 
\label{sub-comm}
Now let use the standard notation in differential geometry, where a vector field is interpreted as derivation of smooth functions \cite{conlon, gadea}. Let
\begin{eqnarray*}
X=\sum\limits_{i=1}^{n}f_{i}\frac{\partial}{\partial x_{i}},\quad
Y=\sum\limits_{i=1}^{n}g_{i}\frac{\partial}{\partial x_{i}}
\end{eqnarray*}
be two vector fields in an open domain of $\mathbb{R}^{n}$. Then the \emph{Lie bracket} of these two vector fields is defined by (\cite{conlon}, Section 2)
\begin{eqnarray}
[X,Y]=\sum\limits_{i=1}^{n}\Big{(}X(g_{i})-Y(f_{i})\Big{)}
\frac{\partial}{\partial x_{i}}.
\label{brack}
\end{eqnarray}
Vector fields commute, if $[X,Y]=0$. For corresponding flows $F$ and $G$, this means that $F(G(\m{x},s),t)=G(F(\m{x},t),s)$, for $s,t$ small enough.\\

In this paper, we amend the items i) through v) of Subsection \ref{sub-motiv} with the following natural problem.
\begin{prob}In dimension $n$, describe all maximal collections of commutative rational vector fields, up to birational equivalence. Can they be explicitly integrated? What about if we limit vector fields to being $s$-homogeneous, $s\in\mathbb{Z}$, and birational transformations as being $1$-homogeneous?
\end{prob}
In this paper we fully solve this problem in case $n=s=2$; see Proposition \ref{level-0} and Theorem \ref{thm-main} in Section \ref{sec:cor}. In the last Section \ref{high-dim} we strengthen this Problem. Explicit integration of commuting vector fields is our new and chief contribution to the subject. This, yet again, emphasizes the need to consider $2$-homogeneous vector fields separately. As noted in (\cite{alkauskas-super1}, Subsection 7.3), any flow with rational vector in $\mathbb{R}^{n}$ can be described in terms of a projective flow in $\mathbb{R}^{n+1}$, so limiting ourselves to projective flows in higher dimensions is not less general approach than to consider any rational vector fields, homogeneous or not. 

\section{Overview}It is impossible to summarize a research on plane vector fields. For example, there exists around $2000$ papers on quadratic vector fields in the plane alone. In this section we briefly touch few questions related to polynomial or rational vector fields, and more thoroughly will delve into the subject related to commutativity. \\

A holomorphic vector field on the manifold $M$ is said to be \emph{complete}, if for every $P\in M$, the solution of the differential equation at $P$ is defined for every complex value of the time variable $t$. For example, the vector field $x^2\frac{\partial}{\partial x}$ is not complete on $\mathbb{R}$, since the flow it generates, $F(x,t)=\frac{x}{1-xt}$, is not defined for $t=x^{-1}$. Meanwhile, a vector field $x\frac{\partial}{\partial x}$ defines a flow $F(x,t)=xe^{t}$, and so is complete in $\mathbb{R}$. In \cite{bustinduy}, the author proves that a complete polynomial (where both coordinates are not simultaneously linear) vector field on $\mathbb{C}^{2}$ has at most one zero. In \cite{JDG} the authors study  a necessary condition for the integrability of the polynomial vector fields in the plane by means of the differential Galois Theory. They employ variational equations around a particular solution to obtained a necessary condition for the existence of a rational first integral.  In \cite{coutinho-m1}, the authors present an algorithm which can be used to check whether a given derivation (vector field) of the complex affine plane has an invariant algebraic curve. In \cite{ferragut}, the remarkable values for polynomial vector fields in the plane having a rational first integral are investigated from a dynamical point of view. If $H=f/g$ is a rational first integral for the polynomial vector field, then $c\in\mathbb{C}\cup\{\infty\}$ is called \emph{the remarkable value of}  $H$, if $f+cg$ is a reducible polynomial in $\mathbb{C}[x,y]$.\\

Let us now turn to papers with Lie bracket in mind. In \cite{petravchuk} the author proves the following result. Let us treat polynomial vector fields as derivations $D:k[x,y]\mapsto k[x,y]$, where the field $k$ is of characteristic $0$. Then if two derivations $D_{1}$ and $D_{2}$ commute (this is the same as saying that vector fields commute), then (i) either they have a common polynomial eigenfunction, i.e. non-constant polynomial $f\in k[x,y]$ such that $D_{1}(f)=\lambda f$, $D_{2}(f)=\mu f$ for some $\lambda,\mu\in k[x,y]$, or (ii) they are Jacobian derivations
\begin{eqnarray*}
D_{1}(g)=D_{u}(g):=\begin{vmatrix}
\frac{\partial u}{\partial x} & \frac{\partial u}{\partial y}\\
\frac{\partial g}{\partial x} & \frac{\partial g}{\partial y}
\end{vmatrix},
\end{eqnarray*}
and $D_{2}(g)=D_{v}(g)$, for all $g\in k[x,y]$. A polynomial eigenfunction of a derivation is also called \emph{a Darboux polynomial}. In \cite{li}, the author generalize the last result, proving that $n$ pairwise commuting derivations of the polynomial
ring in $n$ variables over a field of characteristic
$0$ form a commutative basis of derivations if and only if they
are $k$-linearly independent and have no common Darboux polynomials.\\

In \cite{choud} the authors consider vector fields which are not necessarily polynomial. They use methods of linearization and commutation to tackle the isochronisity problem, and use Darboux polynomials to obtain inverse integrating factor for a ODE system. \emph{Isochronicity} - it is when periodic orbits around the center of a vector field have the same period, like in the vector field $(-y,x)$ case. This is intricately related to \emph{stability}. A a vector field is said to have a \emph{center} at a point $P$, if there exists a punctured neighbourhood of $P$ in which every orbit is a closed non-trivial loop.  The \emph{linearization} is a local diffeomorphism which transforms the system in question into $\{\dot{x}=-y,\dot{y}=x\}$. Several examples of commuting polynomial vector fields are presented. For example,
\begin{eqnarray*}
(-y,x+3xy+x^3),\text{ and }(-x-xy-x^3,-y-y^2+x^2+x^4),
\end{eqnarray*}
also (\emph{the cubic Kolmogorov system})
\begin{eqnarray*}
(-y+2xy-ax^2y,x-x^2+y^2-axy^2)\text{ and }(x-x^2+y^2-axy^2,y-2xy-ay^3),\quad a\in\mathbb{R}.
\end{eqnarray*}
MAPLE packages {\tt DifferentialGeometry} and {\tt LieAlgebras} check commutativity for us.  More about relations of commutators, Lie bracktets to linearization and isochronisity can be found in \cite{freire,maza-1}. In \cite{nagloo}, the authors cite the classical theorem (Theorem 2.1, \cite{nagloo}) which allows to integrate the vector field if one knows another linearly independent field which commutes with it. \emph{A posteriori}, note that we will also encounter few aspects of this method. In particular, we will essentially use the Wr\'{o}nskian of the system (\ref{sys}); see (\ref{wronskian}). \\

In \cite{gine-maza} the authors give a constructive procedure to get the change of variables that orbitally linearizes a smooth planar vector field on around an elementary singular point (or a nilpotent singular point) from a given infinitesimal generator of a Lie symmetry. Recall that a vector field $Y$ is an \emph{infinitesimal generator} of a Lie symmetry of a vector field $X$, if the commutation relation $[X,Y]=\nu(x,y)X$ holds for some smooth scalar function $\nu(x,y)$. Also, the autonomous ODE system $\{\dot{x}=P(x,y),\dot{y}=Q(x,y)\}$ is said to be \emph{orbitally linearizible} at the origin $(0,0)$ (which is assumed to be a singular point), if there exists a smooth-near identity change of coordinates $(u(x,y),v(x,y))=(x+o(x,y),y+o(x,y))$ in the neighbourhood $U\subset\mathbb{C}^{2}$ of the origin, which transforms the initial system into
\begin{eqnarray*}
\left\{\begin{array}{l}
\dot{u}=\lambda u\cdot h(u,v),\\
\dot{v}=\mu v\cdot h(u,v),
\end{array}
\right.
\end{eqnarray*}
$\lambda,\mu\in\mathbb{C}$, and $h(u,v)$ is a smooth scalar function in $U$, $h(0,0)\neq 0$.

\label{review} 
\section{Arithmetic of rational vector fields} In this section we present few examples which show that integration of rational $2$-dimensional $2$-homogeneous vector fields greatly depend on the fine arithmetic structure of vector field itself. We put $\m{x}=(x,y)$.

\subsection{Rational flow} Consider the vector field $(3y^2,y^3x^{-1})$. It can be integrated explicitly, and the outcome is the flow
\begin{eqnarray*}
\phi(x,y)=\big{(}u(x,y),v(x,y)\big{)}=\bigg{(}\frac{(y^2+x)^3}{x^2},\frac{y(y^2+x)}{x}\bigg{)}.
\end{eqnarray*}
This is a rational flow of level $2$. The orbits are curves $v^3u^{-1}=\mathrm{const.}$

\subsection{Algebraic flow}
The flow $\phi(\m{x})=(u(x,y),v(x,y))$ generated by the vector field $(-4x^2+3xy,-2xy+y^2)$ is algebraic and is given by the expression (\cite{alkauskas-ab}, Proposition 2)
\begin{eqnarray*}
\phi(x,y)=\big{(}u(x,y),v(x,y)\big{)}=\left(\frac{y\sqrt{4x+(y-1)^2}+y^2+2x-y}{8x+2(y-1)^2},\frac{y}{\sqrt{4x+(y-1)^2}}\right).
\end{eqnarray*}
The orbits of this flow are genus $0$ curves $u^{-1}(u-v)^{-1}v^{4}=\mathrm{const.}$ So, this is also level $2$ flow.

\subsection{Abelian non-algebraic flow} The flow $\phi(\m{x})=(u(x,y), v(x,y))$ generated by the vector field $(2x^2-4xy,-3xy+y^2)$ is abelian flow and is given by the analytic expression (\cite{alkauskas-ab}, Proposition 4)
\begin{eqnarray}
\phi(\m{x})=
\left(\frac{\m{k}^{4/5}\big{(}\alpha(\frac{x}{y})-\v\big{)}\v}{\Big{(}\m{k}\big{(}\alpha(\frac{x}{y})-\v\big{)}-1\Big{)}^{2/5}}\,,
\frac{\v}{\m{k}^{1/5}\big{(}\alpha(\frac{x}{y})-\v\big{)}\cdot\Big{(}\m{k}\big{(}\alpha(\frac{x}{y})-\v\big{)}-1\Big{)}^{2/5}}\right).
\label{abel}
\end{eqnarray}
Here $\v=\v(x,y)=[x(x-y)^{2}y^{2}]^{1/5}$, $\alpha$ is an abelian integral
\begin{eqnarray*}
\alpha(x)=\frac{1}{5}
\int\limits_{1}^{\frac{1}{1-x}}\frac{\d t}{t^{3/5}(t-1)^{4/5}},
\end{eqnarray*}
and $\m{k}$ is an abelian function, the inverse of $\alpha$. The orbits of this flow are genus $2$ curves $x(x-y)^{2}y^{2}=\mathrm{const.}$ In the special case, one has
\begin{eqnarray*}
\frac{u(x,-x)}{v(x,-x)}=\m{k}\big{(}c-4^{1/5}x\big{)},\quad c=\alpha(-1).
\end{eqnarray*}
The pair of abelian functions $\big{(}\m{k}(z),\m{k}'(z)\big{)}$ parametrizes (locally) the genus $2$ singular curve $5^{5}(1-x)^{3}x^{4}=y^{5}$. In particular, one can give an alternative expression
\begin{eqnarray*}
\phi(\m{x})=
\left(\frac{5^{2/3}\m{k}^{4/3}\big{(}\alpha(\frac{x}{y})-\v\big{)}\v}{\m{k}'\,^{2/3}\big{(}\alpha(\frac{x}{y})-\v\big{)}}\,,
\frac{5^{2/3}\m{k}^{1/3}\big{(}\alpha(\frac{x}{y})-\v\big{)}\v}{\m{k}'\,^{2/3}\big{(}\alpha(\frac{x}{y})-\v\big{)}}\right).
\end{eqnarray*}
As was noted several times in \cite{alkauskas-super1}, using further tools from the theory of abelian functions, like Abel-Jacobi theorem \cite{lang}, it is possible to present the closed-form expression for $\phi$ which involves only abelian functions but \emph{not} abelian integrals. This task is carried out for the icosahedral superflow in detail in \cite{alkauskas-super2,alkauskas-super4}, where after the triple reduction, the flow can be described in terms of abelian functions over the curve of genus $3$. In case the orbits of the flow are elliptic curves (like for the tetrahedral superflow case in \cite{alkauskas-super1}, or an unramified flow with the vector field $(x^2-2xy,y^2-2xy)$ in \cite{alkauskas-un}), this further simplification amounts to application of addition formulas for elliptic functions, either in Weierstrass or Jacobi form. This also applies to the case when the flow can be described in terms of elliptic curves via a reduction of hyperelliptic curves into elliptic ones, like for the octahedral superflow in \cite{alkauskas-super1}. Such reduction itself traces its roots from the works of Legendre in \cite{legendre}.\\

In the present case of the flow (\ref{abel}), the orbits are of genus $2$, and thus the situation is more complicated.
\subsection{Non-abelian flow} The flow $\phi(\m{x})=(u(x,y),v(x,y))$ generated by the vector field $(x^2+xy+y^2,xy+y^2)$ is non-abelian integral flow and is given by the expression (\cite{alkauskas-ab}, Proposition 6)
\begin{eqnarray*}
\big{(}u(x,y),v(x,y)\big{)}=\bigg{(}\psi\v\exp\Big{(}\psi+\frac{\psi^{2}}{2}\Big{)}\,,\v\exp\Big{(}\psi+\frac{\psi^{2}}{2}\Big{)}\bigg{)};\\
\text{ here } 
\psi=\ll\Big{(}\beta\Big{(}\frac{x}{y}\Big{)}-\v\Big{)},
\v=\exp\Bigg{(}-\frac{x}{y}-\frac{x^2}{2y^2}\Bigg{)}y,
\end{eqnarray*}
and where $\beta(x)$ is the error function
\begin{eqnarray*}
\beta(x)=-\frac{\sqrt{\pi e}}{\sqrt{2}}\,\mathrm{erf}\bigg{(}\frac{x+1}{\sqrt{2}}\bigg{)},\quad
\mathrm{erf}(x)=\frac{2}{\sqrt{\pi}}\int\limits_{0}^{x}e^{-t^2}\d t,
\end{eqnarray*} 
and $\ll$ is the inverse of $\beta$. The orbits of this flow are transcendental curves $\v=\mathrm{const.}$\\

This is roughly the arithmetic hierarchy of $2$-dimensional $2$-homogeneous rational vector fields. The majority of other $2$-homogeneous vector fields produce flows whose orbits are non-arithmetic objects. For example, orbits of the vector field $(\sqrt{2}xy,-y^2)$ are curves $uy^{\sqrt{2}}=\mathrm{const.}$
\subsection{$3$-dimensional case}The situation in higher dimensions is much more diverse. In a $3$-dimensional case, the flow might be rational, algebraic, abelian (possessing two rational first integrals), or confined on an algebraic surface (possessing exactly one rational first integral), or possessing no rational first integrals at all. For example, the example of Jouanolou \cite{jouanolou} shows that 
the vector field 
\begin{eqnarray*}
(y^2,z^2,x^2)
\end{eqnarray*}
does not have a rational first integral. For the glimpse into this profound subject, touching algebraic geometry, constants of derivation, foliations, algebraic leaves, and so on, we refer to \cite{maciejewski,ollagnier,nowicki-torun}.

\section{Commutative projective flows}
\label{sec:cor}
\subsection{Main results}
And so, now we investigate rational $2$-homogeneous vector fields which commute. It appears that, apart from level $0$ flows, this is satisfied only for special pairs of algebraic flows of level $1$. This is rather a remarkable fact. Nevertheless, one should expect, at least in a $2$-dimensional case, this kind of fact beforehand, since we can easily imagine that two pairs of $2$-dimensional algebraic functions $(u,v)$ and $(a,b)$ commute, due to some inner dependence, while it is hardly imaginable that this might happen for functions which have the appearance, for example, similar to that of (\ref{abel}). Indeed, commutativity of vector fields implies, among other properties, that for corresponding flows, given by $(u,v)$ and $(a,b)$, the identities $u(a,b)=a(u,v)$, $v(a,b)=b(u,v)$ hold. \\

 As a very simple consequence of the main differential system, we have the following result.  
\begin{prop}
\label{level-0}
Assume that two projective flows $\phi(\m{x})\neq (x,y)$ and $\psi(\m{x})\neq (x,y)$ with rational vector fields $(\varpi,\varrho)$ and $(\alpha,\beta)$, respectively, commute. Assume that $y\varpi-x\varrho=0$. Then $y\alpha-x\beta=0$. Thus, they are both level $0$ rational flows \cite{alkauskas}.\\

Conversely, assume $J(x,y)$ and $K(x,y)$ are arbitrary $1$-homogeneous rational functions, and let us define level $0$ flows
\begin{eqnarray*}
\phi(x,y)=\bigg{(}\frac{x}{1-J(x,y)},\frac{y}{1-J(x,y)}\bigg{)},\quad \psi(x,y)=\bigg{(}\frac{x}{1-K(x,y)},\frac{y}{1-K(x,y)}\bigg{)}.
\end{eqnarray*}    
Then $\phi$ and $\psi$ commute, their vector fields are $(xJ(x,y),yJ(x,y))$ and $(xK(x,y),yK(x,y))$, respectively, and
\begin{eqnarray*}
\phi\circ\psi(x,y)=\bigg{(}\frac{x}{1-J(x,y)-K(x,y)},\frac{y}{1-J(x,y)-K(x,y)}\bigg{)}.
\end{eqnarray*}
\end{prop}

Next, we formulate the main result of this paper.
\begin{thm}
\label{thm-main}
Suppose, two projective flows $\phi(\m{x})\neq (x,y)$ and $\psi(\m{x})\neq (x,y)$ with rational vector fields $(\varpi,\varrho)$ and $(\alpha,\beta)$ commute. Suppose $y\varpi-x\varrho\neq 0$. Then $\phi$ and $\psi$ are level $1$ algebraic flows.\\

Conversely - for any algebraic projective flow $\phi$ of level $1$, there exists another algebraic flow $\psi$ of level $1$, such that all flows, which commute with $\phi$, are given by $\phi^{z}\circ\psi^{w}$, $z,w\in\mathbb{R}$. \\

Practically commuting projective flows can be constructed as follows.\\ 

Let $\mathscr{V}(x,y)\neq cy$ be a $1$-homogeneous rational function. Let us define algebraic functions $a(x,y)$ and $u(x,y)$ from the equations 
\begin{eqnarray}
\mathscr{V}\Big{(}a(x,y),\frac{y}{y+1}\Big{)}&=&\mathscr{V}(x,y),\label{a-def}\\
\frac{\mathscr{V}(x,y)}{1-\mathscr{V}(x,y)}&=&\mathscr{V}\big{(}u(x,y),y\big{)}.\label{u-def}
\end{eqnarray}
Then projective flows
\begin{eqnarray*}
\psi=\Big{(}a(x,y),\frac{y}{y+1}\Big{)}\text{ and }\phi=\big{(}u(x,y),y\big{)}
\end{eqnarray*} 
commute. Any pair of commutative, level $1$ algebraic projective flows can be obtained from these pairs via conjugation with $1$-homogeneous birational plane transformation.\\

Finally, the orbits of the projective flow 
\begin{eqnarray*}
\phi^{z}\circ\psi^{w}=\left(u^{z}\Big{(}a^{w}(x,y),\frac{y}{1+wy}\Big{)},\frac{y}{1+wy}\right)
\end{eqnarray*}
 are given by a $1$-homogeneous (in $x,y$) function
\begin{eqnarray*}
\frac{\mathscr{V}(x,y)y}{z\mathscr{V}(x,y)+wy}=\mathrm{const}.
\end{eqnarray*}
 \end{thm}

To be clear which branch of algebraic function is chosen, we note that the equalities (\ref{a-def}), (\ref{u-def}) and $1$-homogeneity of $\mathscr{V}$ imply that
\begin{eqnarray*}
\mathscr{V}\Big{(}\frac{a(xz,yz)}{z},\frac{y}{yz+1}\Big{)}=\mathscr{V}(x,y),\quad
\frac{\mathscr{V}(x,y)}{1-z\mathscr{V}(x,y)}=\mathscr{V}\Big{(}\frac{u(xz,yz)}{z},y\Big{)}.
\end{eqnarray*}
So, we must choose such branches that are compatible with boundary conditions (\ref{bound}); that is, $\lim_{z\rightarrow 0}\frac{u(xz,yz)}{z}=x$, $\lim_{z\rightarrow 0}\frac{a(xz,yz)}{z}=x$.\\
 
One of the consequences that the flows $(a(x,y),\frac{y}{y+1})$ and $(u(x,y),y)$ commute is the identity among algebraic functions
\begin{eqnarray*}
u\Big{(}a(x,y),\frac{y}{y+1}\Big{)}=a\big{(}u(x,y),y\big{)}.
\end{eqnarray*}
This can be verified easily; below we set $a=a(x,y)$, $u=u(x,y)$:
\begin{eqnarray*}
\mathscr{V}\left(u\Big{(}a,\frac{y}{y+1}\Big{)},\frac{y}{y+1}\right)\mathop{=}^{(\ref{u-def})}\frac{\mathscr{V}(a,\frac{y}{y+1})}{1-\mathscr{V}(a,\frac{y}{y+1})}\mathop{=}^{(\ref{a-def})}
\frac{\mathscr{V}(x,y)}{1-\mathscr{V}(x,y)}\mathop{=}^{(\ref{u-def})}\mathscr{V}(u,y)\mathop{=}^{(\ref{a-def})}\mathscr{V}\Big{(}a(u,y),\frac{y}{y+1}\Big{)}.
\end{eqnarray*}
The correct choice of branches implies the needed identity.\\

We will need the following result, proved in two independent ways in \cite{alkauskas,alkauskas-2}. Birational $1$-homogeneous maps ($1$-BIR for short) $\mathbb{R}^{2}\mapsto\mathbb{R}^{2}$ were described in \cite{alkauskas}. They are either non-degenerate linear maps, either birational maps of the form $\ell_{P,Q}$, given by 
\begin{eqnarray}
\quad\ell_{P,Q}(x,y)=\bigg{(}\frac{xP(x,y)}{Q(x,y)},\frac{yP(x,y)}{Q(x,y)}\bigg{)}, 
\label{bir}
\end{eqnarray}
where $P,Q$ are homogeneous polynomials of the same degree, or are a composition of a linear map and some $\ell_{P,Q}$. By a direct calculation, $\ell^{-1}_{P,Q}(x,y)=\ell_{Q,P}(x,y)$. If $\phi$ is a projective flow with rational vector field, and $\ell$ is a $1$-BIR, then $\ell^{-1}\circ\phi\circ\ell$ is also a projective flow with a rational vector field. Let $v(\phi,\m{x})$ be a vector field of the projective flow $\phi$. \\

The following result describes transformation of the vector field under conjugation with $\ell_{P,Q}$. (Note that we mentioned the general case for such transformations as a crucial ingredient in dealing with rational vector fields in the middle of Subsection \ref{sub-motiv}).  
  
\begin{prop}
\label{prop-trans}
Consider $\ell_{P,Q}$ given by (\ref{bir}), and let $A(x,y)=P(x,y)Q^{-1}(x,y)$, which is a $0$-homogeneous function. Suppose that $v(\phi,\m{x})=(\varpi(x,y),\varrho(x,y))$. Then
\begin{eqnarray*}
v(\ell^{-1}_{P,Q}\circ\phi\circ\ell_{P,Q};\m{x})=\big{(}\varpi_{0}(x,y),\varrho_{0}(x,y)\big{)},
\end{eqnarray*}
where
\begin{eqnarray}
\left\{\begin{array}{c}
\varpi_{0}(x,y)=A(x,y)\varpi(x,y)-A_{y}[x\varrho(x,y)-y\varpi(x,y)],\\
\varrho_{0}(x,y)=A(x,y)\varrho(x,y)+A_{x}[x\varrho(x,y)-y\varpi(x,y)].
\end{array}
\right.
\label{vecconj}
\end{eqnarray}

As a corollary,
\begin{eqnarray*}
x\varrho_{0}(x,y)-y\varpi_{0}(x,y)=A(x,y)[x\varrho(x,y)-y\varpi(x,y)].
\end{eqnarray*} 
 \label{conjug}
\end{prop}

\subsection{The core of the proof}Our main ideas how to prove Theorem \ref{thm-main} are very transparent and can be described immediately as follows. \\

First, we write a condition that vector fields of flows $\phi=(u,v)$ and $\psi=(a,b)$ commute. This is tantamount to the property that Lie bracket of both vector fields vanish. After some transformations we arrive at the linear system of two ODEs. A trivial case aside (Proposition \ref{level-0}), calculation of its Wr\'{o}nskian shows that both flows are in fact abelian flows (flows whose orbits are algebraic curves) of level $1$ - there are two homogeneous rational functions $\mathscr{W}$ and $\mathscr{V}$ of homogeneity degree $1$, such that $\mathscr{W}(u,v)=\mathscr{W}(x,y)$, $\mathscr{V}(a,b)=\mathscr{V}(x,y)$.\\

 But being abelian flow of level $1$, at its turn, implies that with a help of conjugation with a proper $1$-BIR, the second coordinate of the vector field of the flow $\phi$ can be made identically zero - this is possible only for level $1$ abelian flows. Suppose, this holds. Now, the differential system is much more simple, and it implies that the second coordinate of the vector field of the flow $\psi$ is equal to $cy^2$. Since these two vector fields must be non-proportional, it gives $c\neq 0$; for example, after conjugation with a homothety we may assume that $c=-1$. But then the second variable split, and we have $b=\frac{y}{y+1}$. So, without loss of generality, after performing a $1$-BIR conjugation, we can consider $\phi=(u,y)$ and $\psi=(a,\frac{y}{y+1})$. Since $\mathscr{V}(a,\frac{y}{y+1})=\mathscr{V}(x,y)$, this shows that $a$ is an algebraic function, and from symmetry considerations we get that $u$ is algebraic function, too - if a projective flow is $1$-BIR conjugate to algebraic flow, a flow itself must be algebraic.\\

To find when a vector field $(\varpi,\varrho)$ produces algebraic flow of level $1$, we need the following criterion, which follows from results in \cite{alkauskas-2}. 
\begin{prop}
\label{prop-crit}
Let $(\varpi,\varrho)$ is a pair or $2$-homogeneous rational functions, $x\varrho-y\varpi\neq 0$. Let $\varrho(x)=\varrho(x,1)$, $\varpi(x)=\varpi(x,1)$. A projective flow with a vector field $(\varpi,\varrho)$ is a level $1$ algebraic flow if and only if all the solutions of the ODE
\begin{eqnarray*}
f(x)\varrho(x)+f'(x)(x\varrho(x)-\varpi(x))=1
\end{eqnarray*}
are rational functions.
\end{prop}
On the other hand, a flow $\phi$ is algebraic flow of level $N$ if and only if any solution of the same differential equation can be written in the form $f(x)=r(x)+\sigma q^{1/N}(x)$, where $\sigma\in\mathbb{R}$, $r(x)$ and $q(x)$ are rational functions, and positive $N$ is the smallest possible \cite{alkauskas-2}.

\section{The proof}

We will need the following statement, whose proof is immediate.
\begin{prop}
Let $\phi$ and $\psi$ be two projective commuting flows, and $\ell$ be a $1$-BIR. Then the flows $\ell^{-1}\circ\phi\circ\ell$ and $\ell^{-1}\circ\psi\circ\ell$ commute, too.
\label{prop2}
\end{prop}
To find when two projective flows commute, we apply standard results from differential geometry claiming that if vector fields are smooth, this happens exactly when the Lie bracket of both vector fields vanishes \cite{conlon, gadea}. In the particular case of projective flows, we have the following result.

\begin{prop}Let $\phi(\m{x})=(u,v)$ and $\psi(\m{x})=(a,b)$ be two projective flows with smooth vector fields $(\varpi,\varrho)$ and $(\alpha,\beta)$, respectively. The flows $\phi$ and $\psi$ commute if and only if $\phi\circ\psi$ is a projective flow again. This happens exactly when 
\begin{eqnarray*}
\left\{\begin{array}{c}
\varpi_{x}\alpha+\varpi_{y}\beta=\varpi\alpha_{x}+\varrho\alpha_{y},\\
\varrho_{x}\alpha+\varrho_{y}\beta=\varpi\beta_{x}+\varrho\beta_{y}.
\end{array}
\right.
\end{eqnarray*}
\label{prop1}
\end{prop}
All we are left is to calculate the Lie bracket $[(\varpi,\varrho),(\alpha,\beta)]$ of these two vector fields, and equate it to zero. For $n=2$ and $X=\varpi\frac{\partial}{\partial x}+\varrho \frac{\partial}{\partial y}$, $Y=\alpha\frac{\partial}{\partial x}+\beta\frac{\partial}{\partial y}$, the condition $[X,Y]=0$ and (\ref{brack}) gives the statement of the Proposition (the same system is given by Definition 1 in \cite{gine-maza}).\\
 
Let $\phi$, $\psi$ be two commuting projective flows. Since $\alpha$, $\beta$, $\varpi$, $\varrho$ are $2$-homogeneous, Euler's identity gives $x\alpha_{x}+y\alpha_{y}=2\alpha$, similarly for $\beta, \varpi,\varrho$, and so the corresponding system of Proposition \ref{prop1} can be written as 

\begin{eqnarray}
\left\{\begin{array}{r r r}
\alpha_{x}(y\varpi-x\varrho)=&(y\varpi_{x}-2\varrho)\alpha+& (2\varpi-x\varpi_{x})\beta,\\
\beta_{x}(y\varpi-x\varrho)=&y\varrho_{x}\alpha-& x\varrho_{x}\beta.\\
\end{array}
\right.
\label{sys}
\end{eqnarray}
Suppose, the pair $(\varpi,\varrho)$ is known, and the pair $(\alpha,\beta)$ is unknown. Note also that the first equation can be written as
\begin{eqnarray}
\alpha_{y}(y\varpi-x\varrho)=\varpi_{y}(y\alpha-x\beta).
\label{antra}
\end{eqnarray} 
First, suppose $y\varpi-x\varrho=0$. Then the second equation of the system (\ref{sys}) gives $\varrho_{x}(y\alpha-x\beta)=0$. Assume that $y\alpha-x\beta\neq 0$. Then $\varrho_{x}=0$, $\varrho=cy^2$. Equally, (\ref{antra}) gives $\varpi=dx^2$, and $y\varpi-x\varrho=0$ is satisfied only if $c=d=0$, hence we have an identity flow $\phi(\m{x})=(x,y)$. If this is not the case, we must necessarily have $y\alpha-x\beta=0$, and this proves Proposition \ref{level-0}. Henceforth we may assume $y\varpi-x\varrho\neq 0$, $y\alpha-x\beta\neq 0$.\\
  
Recall again that in (\ref{sys}), the pair $(\alpha,\beta)$ is unknown and is to be determined.  One solution of this system is $(\varpi,\varrho)$. Fix $(\alpha,\beta)$ as a linearly independent solution. The trace of this linear system of ODEs is equal to
\begin{eqnarray*}
T(x,y)=\frac{2\varrho+x\varrho_{x}-y\varpi_{x}}{x\varrho-y\varpi}=\frac{\varrho}{x\varrho-y\varpi}+
\frac{\d}{\d x}\ln(x\varrho-y\varpi).
\end{eqnarray*}
It is a $(-1)$-homogeneous function. Using the differential equation (\ref{orbits}), we obtain that the Wr\'{o}nskian of the above system is equal to, according to Liouville's formula,
\begin{eqnarray}
\alpha\varrho-\beta\varpi=\exp\bigg{(}\int T(x,y)\d x\bigg{)}=\mathscr{W}(x,y)\cdot(x\varrho-y\varpi).
\label{wronskian}
\end{eqnarray}
In fact, while integrating, we keep in mind that $x$ is variable, $y$ is constant, only we make sure that the obtained functions have the same degree of homogeneity: indeed, the left hand side is $4$-homogeneous, $(x\varrho-y\varpi)$ is $3$-homogeneous, and $\mathscr{W}(x,y)$ is $1$-homogeneous. So, (\ref{wronskian}) holds up to a factor of a $0$-homogeneous function in $y$; hence, a constant factor. Of course, if the equation $\mathscr{W}(x,y)=\mathrm{const.}$ defines orbits for the flow $\phi$, so does $c\mathscr{W}(x,y)=\mathrm{const.}$ for $c\neq 0$. \\

From symmetry considerations, if $\mathscr{V}(x,y)$ are orbits for the projective flow with a vector field $(\alpha,\beta)$, we get
\begin{eqnarray}
\beta\varpi-\alpha\varrho=\mathscr{V}(x,y)\cdot(x\beta-y\alpha).
\label{wronskian2}
\end{eqnarray}
This, together with (\ref{wronskian}), gives the first crucial corollary:
\begin{itemize}
\item[$\star$]\emph{If two projective flows with rational vector fields commute}, $x\varrho-y\varpi\neq 0$, $x\beta-y\alpha\neq 0$, {they are both abelian flows of level} $1$.
\end{itemize}
From (\ref{wronskian}), expressing $\alpha$ and substituting into the second equation of (\ref{sys}), we obtain:
\begin{eqnarray*}
\beta_{x}(y\varpi-x\varrho)=y\varrho_{x}\frac{\beta\varpi}{\varrho}+
y\varrho_{x}\frac{x\varrho-y\varpi}{\varrho}\mathscr{W}
-x\varrho_{x}\beta.
\end{eqnarray*}
Or, simplifying, we obtain a non-homogeneous first order linear ODE for the function $\beta$:  
\begin{eqnarray}
\beta_{x}\varrho=\varrho_{x}\beta-
y\varrho_{x}\mathscr{W}.
\label{ess}
\end{eqnarray}
If $\varrho\neq 0$, solving this linear ODE, we obtain a solution
\begin{eqnarray}
\beta=-\varrho\int\frac{y\varrho_{x}\mathscr{W}}{\varrho^{2}}\d x.
\label{beta-ok}
\end{eqnarray}
(The constant of integration is a $0$-homogeneous function in $y$; hence, a constant). We will not need this formula in the proof of the main Theorem, but it has an advantage that if a vector field $(\varpi,\varrho)$ is known, so we know equation for the orbits $\mathscr{W}$, then the above allows to find uniquely $\beta$, up to a summand proportional to $\varrho$; this will be used in Section \ref{examples} where several examples are given. This also shows that if $\varrho\neq 0$, there exists one vector field $(\alpha,\beta)$ such that all vector fields that commute with $(\varpi,\varrho)$ are given by $z(\varpi,\varrho)+w(\alpha,\beta)$. The same conclusion follows if $\varpi\neq 0$. Of course, $(\varpi,\varrho)=(0,0)$ holds only for the identity flow $(x,y)$. \\

Now we will make one significant simplification, minding that the flows are abelian flows of level $1$. Let us define
\begin{eqnarray*}
A(x,y)=\frac{y}{\mathscr{W}(x,y)}.
\end{eqnarray*}
This is a $0$-homogeneous function.  We know that $\mathscr{W}$ satisfies (\ref{orbits}). Now, consider a $1$-BIR given by $\ell(x,y)=(xA,yA)$. Let us use Proposition \ref{prop-trans}. This shows that a second coordinate of the flow $\ell^{-1}\circ\phi\circ\ell$ with rational vector field is identically equal to $0$. Here we used essentially the fact that the orbits are given by $1$-homogeneous rational functions. \\

Thus, let now consider two commuting flows $\ell^{-1}\circ\phi\circ\ell$ and $\ell^{-1}\circ\psi\circ\ell$. If we are interested in these flows up to $1$-BIR conjugacy, we can, without the loss of generality, consider $\varrho=0$, $\varpi\neq 0$, $(u,v)=(u(x,y),y)$. \\

In this case, the system of Proposition \ref{prop1} gives:
\begin{eqnarray*}
\left\{\begin{array}{c}
\varpi_{x}\alpha+\varpi_{y}\beta=\varpi\alpha_{x},\\
\beta_{x}=0.
\end{array}
\right.
\end{eqnarray*}If $\beta=0$, this gives $\alpha=C\varpi$, and we know that the flow commutes with itself. So let, without the loss of generality, take $\beta=-y^2$. So, $(a,b)=(a,\frac{y}{y+1})$. But then the flow conservation property gives (\ref{a-def}); that is,
\begin{eqnarray*}
\mathscr{V}\Big{(}a,\frac{y}{y+1}\Big{)}=\mathscr{V}(x,y).
\end{eqnarray*}
(See also \cite{alkauskas-2}). Therefore, $a$ is an algebraic function, and $(a,b)$ is an algebraic flow. Yet again from symmetry considerations, $(u,v)$ is also an algebraic flow. Indeed, $1$-BIR conjugation does not impact on property of the flow being algebraic. Hence we have proved the first part of the main Theorem \ref{thm-main}.\\

We will prove the second part and at the same show how to practically produce commuting algebraic flows. As we know from \cite{alkauskas-2}, any algebraic flow is $1$-BIR conjugate to the flow $(a(x,y),\frac{y}{y+1})$ for $a$ algebraic. Now, let $\mathscr{V}(x,y)\neq cy$ be any $1$-homogeneous rational function. Let us define $a(x,y)$ from the equation (\ref{a-def}). Then, if we choose the correct branch (as explained after the Theorem), $(a(x,y),\frac{y}{y+1})$ is a projective flow. Its vector field $(\alpha,\beta)$ satisfies $\mathscr{V}_{x}\alpha+\mathscr{V}_{y}\beta=0$, and so is given by
\begin{eqnarray}
(\alpha,\beta)=\Big{(}\frac{y^2\mathscr{V}_{y}}{\mathscr{V}_{x}},-y^2\Big{)}=
\Big{(}\frac{y\mathscr{V}-xy\mathscr{V}_{x}}{\mathscr{V}_{x}},-y^2\Big{)}=
\Big{(}\frac{y\mathscr{V}}{\mathscr{V}_{x}}-xy,-y^2\Big{)}.
\label{taip}
\end{eqnarray} 
Let now a non-proportional vector field $(\varpi,\varrho)$ commutes with $(\alpha,\beta)$. Similarly as (\ref{ess}), we can prove
\begin{eqnarray*}
\beta\varrho_{x}=\varrho\beta_{x}-
y\beta_{x}\mathscr{V}.
\end{eqnarray*}
This follows easily in the same way (express $\alpha$ from (\ref{wronskian2}) and plug into the second equation of (\ref{sys})), or just due to symmetry considerations, minding the ODE (\ref{ess}). Now, $\beta=-y^2$, and this gives $\varrho_{x}=0$. This implies $\varrho=cy^2$. Next, the vector field $(\varpi,\varrho)+c(\alpha,\beta)$ also commutes with $(\alpha,\beta)$, so we may assume, without the loss of generality, $\varrho=0$.\\

Now, we can find $\varpi$ from (\ref{wronskian2}) and (\ref{taip}):
\begin{eqnarray}
\varpi=\frac{\mathscr{V}(x\beta-y\alpha)+\alpha\varrho}{\beta}=\mathscr{V}\cdot\Big{(}x+\frac{\alpha}{y}\Big{)}\mathop{=}^{(\ref{taip})}\mathscr{V}\cdot\Big{(}x+\frac{\mathscr{V}}{\mathscr{V}_{x}}-x\Big{)}=\frac{\mathscr{V}^2}{\mathscr{V}_{x}}.
\label{varpi}
\end{eqnarray}
This gives a practical way to produce algebraic projective flows which commute. Moreover, we can finish integrating the vector field $(\varpi,\varrho)$ in explicit terms, since we have at our disposition the method to integrate any vector field $(\varpi,0)$. Its integral is a flow $(u(x,y),y)$, where $u$ is found from the equation (\cite{alkauskas}, p. 307)
\begin{eqnarray*}
\int\limits^{\frac{y}{x}}_{\frac{y}{u(x,y)}}\frac{\d t}{\varpi(1,t)}=y.
\end{eqnarray*} 
In this integral, let us make a change $t\mapsto\frac{1}{t}$. Since $\varpi$ is $2$-homogeneous, this gives
\begin{eqnarray*}
\int\limits^{\frac{u(x,y)}{y}}_{\frac{x}{y}}\frac{\d t}{\varpi(t,1)}=y.
\end{eqnarray*} 
Now, let us use (\ref{varpi}). This gives
\begin{eqnarray*}
\int\frac{\d t}{\varpi(t,1)}=\int\limits\frac{\mathscr{V}_{x}(t,1)\d t}{\mathscr{V}^{2}(t,1)}=-\frac{1}{\mathscr{V}(t,1)}.
\end{eqnarray*}
So,
\begin{eqnarray*}
\frac{1}{\mathscr{V}(\frac{x}{y},1)}-\frac{1}{\mathscr{V}(\frac{u}{y},1)}=y.
\end{eqnarray*}
Since $\mathscr{V}$ is $1$-homogeneous, this finally gives the equation for $u$, as given by (\ref{u-def}); that is,
\begin{eqnarray*}
\frac{\mathscr{V}(x,y)}{1-\mathscr{V}(x,y)}=\mathscr{V}(u,y).
\end{eqnarray*}  
Thus, this gives the explicit formulas in Theorem \ref{thm-main}. Also, we can double-check that a vector field is the correct one. Indeed, the last equation can be rewritten as, after plugging $(x,y)\mapsto(xz,yz)$ and dividing by $z$,
\begin{eqnarray*}
\frac{\mathscr{V}(x,y)}{1-z\mathscr{V}(x,y)}=\mathscr{V}\Big{(}\frac{u(xz,yz)}{z},y\Big{)}.
\end{eqnarray*} 

Now differentiation with respect to $z$ and afterwards substitution $z=0$ gives, minding the formula (\ref{vec}), the correct value for the first coordinate of the vector field, given by (\ref{varpi}).\\

Finally, a vector field of the flow $\phi^{z}\circ\psi^{w}$ is equal to
\begin{eqnarray*}
z(\varpi,\varrho)+w(\alpha,\beta)=\Big{(}\frac{z\mathscr{V}^2}{\mathscr{V}_{x}}+\frac{wy\mathscr{V}}{\mathscr{V}_{x}}-wxy,-wy^2\Big{)}=
(\widehat{\varpi},\widehat{\varrho}).
\end{eqnarray*}
Let $\widehat{\mathscr{W}}$ be the equation for the orbits of this flow. We are left to solve the ODE (\ref{orbits}) for a vector field $(\widehat{\varpi},\widehat{\varrho})$. In this case, it reads
\begin{eqnarray*}
\frac{\widehat{\mathscr{W}_{x}}}{\widehat{\mathscr{W}}}=
\frac{wy\mathscr{V}_{x}}{z\mathscr{V}^2+wy\mathscr{V}}=
\frac{\mathscr{V}_{x}}{\mathscr{V}}-\frac{\mathscr{V}_{x}}{\mathscr{V}+\frac{wy}{z}}.
\end{eqnarray*}
Thus,
\begin{eqnarray*}
\widehat{\mathscr{W}}=\frac{\mathscr{V}y}{z\mathscr{V}+wy}.
\end{eqnarray*}
While integrating, we keep in mind that $\widehat{\mathscr{W}}$ is $1$-homogeneous function in $(x,y)$, so integration constant is chosen to be $\ln y-\ln z$. For $(z,w)=(1,0)$ we recover orbits of the flow $\phi$ ($y=\mathrm{const.}$), and for $(z,w)=(0,1)$ we get orbits of the flow $\psi$ ($\mathscr{V}(x,y)=\mathrm{const}$). This finishes the proof of Theorem \ref{thm-main}. 
\section{Examples}
\label{examples}
\subsection{Monomials}The simplest case of a $1$-homogeneous function $\mathscr{V}$ in Theorem \ref{thm-main} is a monomial. So let, in the setting of second half of Theorem \ref{thm-main},
\begin{eqnarray*}
\mathscr{V}(x,y)=x^{n+1}y^{-n},\quad n\in\mathbb{Z}\setminus\{-1\}.
\end{eqnarray*}
This gives
\begin{eqnarray*}
\psi(\m{x})=\Big{(}x(y+1)^{-\frac{n}{n+1}},\frac{y}{y+1}\Big{)},\quad \phi(\m{x})=
\left(\frac{x}{(1-x^{n+1}y^{-n})^{\frac{1}{n+1}}}\,, y\right).
\end{eqnarray*}

We can check by hand that these two are indeed commutative flows; that is, they satisfy (\ref{funk}), and do commute. 
\subsection{Superflow}The flow
\begin{eqnarray*}
\phi(\m{x})=\big{(}x+(x-y)^2,y+(x-y)^2\big{)}
\end{eqnarray*} 
is rational, hence algebraic, and its orbits are curves $x-y=\mathrm{const}$. We give this particular example because this flow has many fascinating properties, related to finite linear groups, infinite linear groups and their Lie algebras. This flow is the simplest example of a reducible $2$-dimensional superflow \cite{alkauskas-super1, alkauskas-super2,alkauskas-super3}. More precisely, the flow has a $6$-fold cyclic symmetry generated by the order $6$ matrix
\begin{eqnarray*}
\gamma=\begin{pmatrix}
\zeta & 0\\
\zeta+\zeta^{-1}&-\zeta^{-1}
\end{pmatrix},\quad \zeta=e^{\frac{2\pi i}{3}}.
\end{eqnarray*}
This show that a flow is a superflow. The full group of symmetries of this flow is infinite \cite{alkauskas-super3}. Since orbits are $1$-homogeneous functions, we can apply the main Theorem. Thus, let $(\varpi,\varrho)=((x-y)^2,(x-y)^2)$, and $\mathscr{W}=x-y$. Formula (\ref{beta-ok}) gives $\beta=2xy-y^2$, and formula (\ref{wronskian}) gives $\alpha=x^2-y^2$. This vector field can be easily integrated using methods from \cite{alkauskas, alkauskas-2}, and the flow we obtain is given by
\begin{eqnarray*}
\psi(\m{x})=\bigg{(}\frac{x-(x-y)^2}{(x-y-1)^2},\frac{y}{(x-y-1)^2}\bigg{)}.
\end{eqnarray*}  
By a direct calculation, these two flows indeed commute, since
\begin{eqnarray*}
\psi^{w}\circ\phi^{z}(\m{x})=\bigg{(}\frac{x+(z-w)(x-y)^2}{(wx-wy-1)^2},\frac{y+z(x-y)^2}{(wx-wy-1)^2}\bigg{)}=\phi^{z}\circ\psi^{w}(\m{x}).
\end{eqnarray*}
\subsection{A quadratic example} Let $(\varpi,\varrho)=(2x^2-3xy, xy-2y^{2})$. This vector field satisfies the condition of Proposition \ref{prop-crit}. In fact, in \cite{alkauskas-ab} we classified all pairs of quadratic forms which produce algebraic flows (see also correcting remarks in \cite{alkauskas-2}), and this particular case corresponds to a pair $(n,Q)=(1,-\frac{1}{2})$. Since the numerator of $Q$ is $1$, this is a flow of level $1$. Thus, we have
\begin{eqnarray*}
\mathscr{W}(x,y)=x^{-2}(x-y)y^{2},\quad(\alpha,\beta)&=&\Big{(}\frac{y^3}{x},\frac{y^{3}}{x}\Big{)}.
\end{eqnarray*}
The system of Proposition \ref{prop1} is satisfied. The method to integrate vector fields with both coordinates proportional is developed in (\cite{alkauskas}, Subsection 4.2). The integral of the vector field $(\alpha,\beta)$ is the flow
\begin{eqnarray*}
(a,b)=\Bigg{(}\frac{(x-y)\sqrt{x^2-2xy^2+2y^{3}}}{\sqrt{x^2-2xy^2+2y^{3}}-y},\frac{xy-y^{2}}{\sqrt{x^2-2xy^2+2y^{3}}-y}\Bigg{)},\quad \mathscr{V}(x,y)=x-y.
\end{eqnarray*}
Again we double-check that initial conditions (\ref{init}) and PDEs (\ref{pde}) are satisfied. To find $u,v$ we use algebraic identities of Theorem in \cite{alkauskas-2}. This gives
\begin{eqnarray*}
\frac{(u-v)v^{2}}{u^{2}}=\frac{(x-y)y^{2}}{x^{2}},\quad
\frac{1}{v}+\frac{u}{v^2}=\frac{1}{y}+\frac{x}{y^2}+2.
\end{eqnarray*}
Some handwork gives the solution
\begin{eqnarray*}
(u,v)=\left(\frac{xy\sqrt{1-2x+2y}+x^2}{(x+y+2y^2)(1-2x+2y)},\frac{y^2\sqrt{1-2x+2y}+xy}{(x+y+2y^2)\sqrt{1-2x+2y}}\right).
\end{eqnarray*}  
Once again, we double check that the PDEs (\ref{pde}) and initial conditions (\ref{init}) are satisfied. It is not that straightforward to check that indeed, these two flows do commute! As mentioned in the introduction, MAPLE confirms this. For the convenience of the reader, MAPLE code to verify these claims are contained in \cite{prop-symb}. Thus,
as a particular example, we prove the following.
\begin{prop}The flow
\begin{eqnarray*}
\phi(\m{x})=\left(\frac{xy\sqrt{1-2x+2y}+x^2}{(x+y+2y^2)(1-2x+2y)}
,\frac{y^2\sqrt{1-2x+2y}+xy}{(x+y+2y^2)\sqrt{1-2x+2y}}\right),
\end{eqnarray*}
with the vector field  $(2x^2-3xy,xy-2y^{2})$ and orbits $u^{-2}(u-v)v^2=\mathrm{const.}$, and the flow
\begin{eqnarray*}
\psi(\m{x})=\left(\frac{(x-y)\sqrt{x^2-2xy^2+2y^{3}}}{\sqrt{x^2-2xy^2+2y^{3}}-y},
\frac{xy-y^{2}}{\sqrt{x^2-2xy^2+2y^{3}}-y}\right),
\end{eqnarray*}
with the vector field $(\frac{y^3}{x},\frac{y^3}{x})$ and orbits $u-v=\mathrm{const.}$, commute. All projective flows which commute with $\phi$ are given by $\phi^{z}\circ\psi^{w}$, $z,w\in\mathbb{R}$. The orbits of the flow $(U,V)=\phi^{z}\circ\psi^{w}$ are given by
\begin{eqnarray*}
\frac{(U-V)V^2}{zU^2-wV^2}=\frac{(x-y)y^2}{zx^2-wy^2}=\mathrm{const.}
\end{eqnarray*}
\label{prop-ex1}
\end{prop}
The statement that the orbits are, for example, $u^{-2}(u-v)v^2=\mathrm{const.}$, is a slight abuse of notation which means the following. Let $u(xz,yz)z^{-1}=u^{z}$, $v(xz,yz)z^{-1}=v^{z}$. What we mean is that
\begin{eqnarray*}
(u^{z})^{-2}(u^{z}-v^{z})(v^{z})^{2}\equiv x^{-2}(x-y)y^{2},
\end{eqnarray*}
and is independent of $z$.
The flow $\phi$ is locally well-defined - if $(x,y)$ is replaced with $(zx,zy)$, where $z$ is sufficiently small, we take the branch of the square root which is equal to $x$ at $z=0$. For the flow $\psi$ we assume
\begin{eqnarray*}
\sqrt{x^2-2xy^2+2y^{3}}=x\sqrt{1-\frac{2y^2}{x}+\frac{2y^{3}}{x^2}},
\end{eqnarray*} 
and similar convention holds for $(x,y)\mapsto(zx,zy)$. 
 
\subsection{A cubic example}Consider the vector field $(2x^2+xy,xy+2y^2)$. Yet again, this vector field satisfies the conditions of Proposition \ref{prop-crit} (MAPLE does it for us), so it produces algebraic flow of level $1$. Indeed, in the setting of (\cite{alkauskas-ab}, Theorem 3), $(n,Q)=(1,\frac{1}{2})$. Moreover, this vector field is symmetric with respect to conjugation with a $1$-BIR involution $i(x,y)=(y,x)$, and so is the flow. \\

Let $A=\frac{y(y-3x)}{6x^2}$. Then formulas in Proposition \ref{conjug} give a vector field
\begin{eqnarray*}
\bigg{(}\frac{4xy^2-y^3-9x^2y}{6x},-y^2\bigg{)}.
\end{eqnarray*} 
The orbits of the flow with the latter vector field are given by
\begin{eqnarray*}
\mathscr{V}(x,y)=\frac{(3x-y)y^3}{(x-y)^3}=\mathrm{const.}
\end{eqnarray*}
So, for $\mathscr{V}$ is given as above, we are now in the position of the second part of Theorem \ref{thm-main}, and can give the final answer in terms of Cardano formulas. Returning back to the vector field $(2x^2+xy, xy+2y^2)$, that is, performing backwards $1$-BIR $(xA^{-1},yA^{-1})$, gives the following result.
\begin{prop}
\label{prop-7}
The flow $\phi(x,y)=\big{(}u(x,y),v(x,y)\big{)}$, where
\begin{eqnarray*}
u(x,y)=\frac{\Bigg{(}\sqrt[3]{x+y\sqrt{\frac{y-3x+6x^2}{x-3y+6y^2}}}+\sqrt[3]{x-y\sqrt{\frac{y-3x+6x^2}{x-3y+6y^2}}}\Bigg{)}x(x-y)}{\sqrt[3]{x-3y+6y^2}\cdot(y-3x+6x^2)}
-\frac{2x^2}{y-3x+6x^2},
\end{eqnarray*}
and $v(x,y)=u(y,x)$, with the vector field $(2x^2+xy,xy+2y^2)$, and orbits $u^2(u-v)^{-3}v^2=\mathrm{const.}$, and the flow $\psi(\m{x})=\big{(}a(x,y),b(x,y)\big{)}$, where the third degree algebraic functions $a,b$ are found from
\begin{eqnarray*}
\frac{(9ax-8x^2-3ay)x^2y^2}{a(ay-3ax+3x^2)^2}=x-3y-6y^2,
\quad b=\frac{a^2(y-3x)}{x^2}+3a,
\end{eqnarray*}
(we choose the branch of the function $``a"$ which satisfies the boundary condition) with the vector field  
\begin{eqnarray}
(\alpha,\beta)=\left(-\frac{xy^3}{(x-y)^2},\frac{3xy^3-2y^4}{(x-y)^2}\right),
\label{al-be}
\end{eqnarray}
 and orbits $\frac{a^2}{3a-b}=\mathrm{const.}$, commute. All projective flows which commute with $\phi$ are given by $\phi^{z}\circ\psi^{w}$, $z,w\in\mathbb{R}$. The orbits of the flow $(U,V)=\phi^{z}\circ\psi^{w}$ are given by
\begin{eqnarray*}
\frac{U^{2}V^{2}}{z(V-U)^{3}+wV^{2}(3U-V)}
=\frac{x^{2}y^{2}}{z(y-x)^{3}+wy^{2}(3x-y)}=\mathrm{const.}
\end{eqnarray*}
\label{prop-ex2}
\end{prop}
The MAPLE code which again checks the last Proposition can be found at \cite{prop-symb}. As mentioned, $v(x,y)=u(y,x)$. MAPLE formally verifies that the PDE (\ref{pde}) is satisfied without even specifying which of the branches of radicals we are using. Also the flow conservation property is satisfied. However, it is known that if a third degree polynomial has three distinct real roots, then Cardano formulas must involve complex numbers (\cite{waerden}, chapter on Galois theory). So, the choice of the branch and thus making sure that the boundary conditions (\ref{init}) are satisfied are not so explicit if one uses the expression for $u$ given in Proposition \ref{prop-7}. Numerical computations show that for $x-3y>0$, $y-3x>0$, the boundary conditions are satisfied if we use positive value for the square root and real values for cubic roots, due to an algebraic identity:
\begin{eqnarray*}
\lim\limits_{z\rightarrow 0}\frac{u(xz,yz)}{z}=\frac{\Bigg{(}\sqrt[3]{x+y\sqrt{\frac{y-3x}{x-3y}}}+\sqrt[3]{x-y\sqrt{\frac{y-3x}{x-3y}}}\Bigg{)}x(x-y)}{\sqrt[3]{x-3y}\cdot(y-3x)}-\frac{2x^2}{y-3x}\equiv x.
\end{eqnarray*}  
Of course, such identities should not come as a surprise: they arise all the time when a cubic polynomial has a root we know in advance - for example, write Cardano formulas for a polynomial $(x-1)(x^2+px+q)$. \\

The algebraic function $a(x,y)$ satisifes the third degree equation
\begin{eqnarray}
F(a,x,y):=(9ax-8x^2-3ay)x^2y^2+(3y+6y^2-x)a(ay-3ax+3x^2)^2=0,
\label{impli}
\end{eqnarray}
and the function $b$ is given by
\begin{eqnarray*}
b=\frac{a^2(y-3x)}{x^2}+3a;
\end{eqnarray*}
the latter follows from the flow conservation property. The polynomial $F(a,x,y)$ is irreducible over $\mathbb{C}[a,x,y]$. To double-check that the integral of the vector field $(\alpha,\beta)$ is $(a,b)$, we will now verify the boundary condition (\ref{init}) and the PDE (\ref{pde}).
Let $a^{z}(x,y)=\frac{a(xz,yz)}{z}$. Then putting $(x,y)\mapsto(xz,yz)$ in (\ref{impli}), we obtain:
\begin{eqnarray*}
(9a^{z}x-8x^2-3a^{z}y)x^2y^2=(x-3y-6y^2z)a^{z}(a^{z}y-3a^{z}x+3x^2)^2.
\end{eqnarray*}
Since for $z=0$, $a^{0}=x$ satisfies the above, we choose the branch for $a^{z}$ such that for $z=0$, $a^{z}=x$.  We are left to verify the PDE for $a$:
\begin{eqnarray*}
a_{x}(x,y)(\alpha(x,y)-x)+a_{y}(x,y)(\beta(x,y)-y)=-a.
\end{eqnarray*}
From (\ref{impli}), we have
\begin{eqnarray*}
a_{x}=-\frac{F_{x}}{F_{a}},\quad a_{y}=-\frac{F_{y}}{F_{a}}.
\end{eqnarray*}
So, the following as if must hold:
\begin{eqnarray*}
-\frac{F_{x}(a,x,y)}{F_{a}(a,x,y)}\cdot\big{(}\alpha(x,y)-x\big{)}
-\frac{F_{y}(a,x,y)}{F_{a}(a,x,y)}\cdot\big{(}\beta(x,y)-y\big{)}+a\mathop{=}^{?}0.
\end{eqnarray*}
However, calculations with MAPLE show that the left hand side of the above, which is a rational function in three variables $a,x,y$, is not identically zero - it is quite a lengthy expression that involves the powers of $a$ up to $a^{5}$. Nevertheless, $a$ is algebraically dependent on $x,y$, and the numerator can be reduced. And indeed, MAPLE confirms that
\begin{eqnarray*}
-\frac{F_{x}}{F_{a}}(\alpha-x)-\frac{F_{y}}{F_{a}}(\beta-y)+a\equiv0\text{ mod }F(a,x,y).
\end{eqnarray*}
Note that we have already encountered an analogous phenomenon in \cite{alkauskas-un} while dealing with elliptic flows. So, this implicitly verifies the PDE (\ref{pde}).\\

One last remark. The vector field $(\varpi,\varrho)=(2x^2+xy,xy+2y^2)$ is invariant under conjugation with a linear involution $i_{0}(x,y)=(y,x)$. This shows that if a vector field $(\alpha,\beta)$, given by (\ref{al-be}), commutes with $(\varpi,\varrho)$, so does a vector field $i_{0}\circ(\alpha,\beta)\circ i_{0}(x,y)$. However, this vector field is not equal to $(\alpha,\beta)$. There is no contradiction to uniqueness property, since we know that there must exist $c,d\in\mathbb{R}$ such that
\begin{eqnarray*}
c(2x^2+xy,xy+2y^2)+d\bigg{(}-\frac{xy^3}{(x-y)^2},\frac{3xy^3-2y^4}{(x-y)^2}\bigg{)}=
\bigg{(}\frac{3x^3y-2x^4}{(x-y)^2},-\frac{x^3y}{(x-y)^2}\bigg{)}.
\end{eqnarray*}
This indeed holds for $c=d=-1$.

\section{Higher dimensions}
\label{high-dim}
The problem of describing various aspects of projective flows in dimension $n\geq 3$, starting from continuous (not necessarily smooth) flows on a single point compactification of $\mathbb{R}^{n}$, symmetry, quasi-flows, rational flows, flows over finite fields, abelian, algebraic and integral flows, are all open. See \cite{alkauskas-un, alkauskas-ab} for a list of 10 problems. This list is continued in \cite{ alkauskas-2, alkauskas-super1,alkauskas-super2,alkauskas-super3}; note that the theory of superflows in dimension $3$ is close to a finish. In relation to the topic of the current paper, in Subsection \ref{sub-comm} we posed another problem, which we will now strengthen.\\

Based on results in dimension $2$, we can construct families of $n$ pairwise commuting projective flows with rational vector fields. Indeed, we will illustrate this in dimension $3$, and the same construction - direct sum of flows - works in any dimension.
\subsection{Extension of a commuting pair of flows}Let $\phi=(u,v)$ and $\psi=(a,b)$ be $2$-dimensional commuting algebraic flows. Then let us define the set of flows
 \begin{eqnarray}
\left\{\begin{array}{l}
\phi(x,y,z)=\big{(}u(x,y),v(x,y),z\big{)},\\
\psi(x,y,z)=\big{(}a(x,y),b(x,y),z\big{)},\\
\xi(x,y,z)=\big{(}x,y,\frac{z}{1-z}\big{)}.\\
\end{array}
\label{a-bit-more}
\right.
\end{eqnarray}
(Here $z$, of course, is no longer a ``time" parameter. The use of the same notation $``\phi"$ and $``\psi"$ for $2$-dimensional flows and their $3$-dimensional extensions should not cause a confusion). These three are algebraic pairwise commuting flows. In particular, as the most basic example, let us consider
 \begin{eqnarray*}
\left\{\begin{array}{l}
\phi(x,y,z)=\big{(}\frac{x}{1-x},y,z\big{)},\\
\psi(x,y,z)=\big{(}x,\frac{y}{1-y},z\big{)},\\
\xi(x,y,z)=\big{(}x,y,\frac{z}{1-z}\big{)}.\\
\end{array}
\right.
\end{eqnarray*}
A very simple argument shows that this is the maximal collection - any flow which commutes with all three is necessarily of the form
\begin{eqnarray*}
\eta(x,y,z)=\Big{(}\frac{x}{1-px},\frac{y}{1-qy},\frac{z}{1-rz}\Big{)}=\phi^{p}\circ\psi^{q}\circ\xi^{r},\quad p,q,r\in\mathbb{R}\text{ are fixed}.
\end{eqnarray*} 
Indeed, suppose a smooth flow
\begin{eqnarray*}
\eta=\big{(}u(x,y,z),v(x,y,z),t(x,y,z)\big{)}
\end{eqnarray*} 
  with a rational $2$-homogeneous vector field commutes with $\phi$, $\psi$, and $\xi$. Then  $\eta$ commutes with $\xi^{r}$ for any $r\in\mathbb{R}$ - vanishing of Lie brackets is a homogeneous condition on vector fields. Writing down, this means that
\begin{eqnarray*}
u\Big{(}x,y,\frac{z}{1-rz}\Big{)}=u(x,y,z),\text{ for any }r.
\end{eqnarray*}
So, $u$ is independent of $z$. Equally, $u$ is independent of $y$, and so $u(x,y,z)=\frac{x}{1-px}$ for a certain $p\in\mathbb{R}$. The same reasoning works for the functions $v$ and $t$.\\

 A bit more generally, a similar argument shows that (\ref{a-bit-more}) is a maximal collection. Indeed, for any  $3$-dimensional projective flow $\eta$ which commutes with $\xi$, the first two coordinates of $\eta$ must be independent of $z$, and therefore the first two coordinates of the vector field of $\eta$, based on the results of the current paper, are a linear combination of vector fields of $2$-dimensional flows  $(u,v)$ and $(a, b)$.

\subsection{Commutative rational flows}Another family of flows was given in \cite{alkauskas-t} in relation to continuous flows on a single point compactification of $\mathbb{R}^{n}$. With commutativity in mind, we can now shed a new light on this particular example.\\

 Indeed, let $Q(\m{x})$ be an arbitrary non-zero quadratic form in $n$ variables with real coefficients, and let $B(\m{x},\m{y})=Q(\m{x}+\m{y})-Q(\m{x})-Q(\m{y})$ be the associated bilinear form. Then we know that for any $\m{a}\in\mathbb{R}^{n}$, the $n$-dimensional rational function
\begin{eqnarray*}
\phi_{\m{a},Q}(\m{x})=\frac{\m{a}Q(\m{x})+\m{x}}{Q(\m{x})\cdot Q(\m{a})+B(\m{x},\m{a})+1}
\end{eqnarray*} 
(numerator is a vector, denominator is a scalar) is a projective flow with a vector field $\m{a}Q(\m{x})-\m{x}B(\m{x},\m{a})$ \cite{alkauskas-t}. Moreover (\cite{alkauskas-t}, Proposition 2),
\begin{eqnarray*}
\phi_{\m{a},Q}\circ\phi_{\m{b},Q}(\m{x})=\phi_{\m{a}+\m{b},Q}(\m{x}).
\end{eqnarray*}
So, for $Q$ fixed and $\m{a}$ varying, these flows mutually commute, and their composition produces a new projective flow, in correspondence with the results of the current paper. Since all possible vectors $\m{a}$ form a vector space of dimension $n$, there are exactly $n$ linearly independent vector fields in this family. \\

Further, let $J(\m{x})$ be a $1$-homogeneous rational function in dimension $n$. Then the flow
\begin{eqnarray*}
\phi_{J}(\m{x})=\frac{\m{x}}{1-J(\m{x})}
\end{eqnarray*} 
(once again, numerator is a vector, denominator is a scalar) is a direct analogue of level $0$ rational flows. Such a flow can be called \emph{level $0$ rational flow in dimension $n$}, and all such flows commute; see Proposition \ref{level-0}.\\
 
With all these examples in mind, we therefore strengthen the problem posed in the end of Subsection \ref{sub-comm} as follows. 
\begin{prob}Let $n\geq 3$ be an integer. Describe maximal sets of pairwise commuting smooth projective flows with rational vector fields in dimension $n$. Is it following true? If in this set there exists at least one flow not of the form $\phi_{J}(\m{x})$, and the set is $n$-dimensional, then this set is generated by $n$ linearly independent vector fields, and all flows in this set can be explicitly integrated in terms of algebraic functions.
\end{prob}
\section{Appendix - Maple codes for propositions 6 and 7}
\begin{verbatim}
> restart: Proposition 6
> alpha:=y^3/x: beta:=y^3/x: 
Vector field of the flow psi.
> a:=(x-y)*x*sqrt(1-2*y^2/x+2*y^3/x^2)/(x*sqrt(1-2*y^2/x+2*y^3/x^2)-y): 
  b:=(x*y-y^2)/(x*sqrt(1-2*y^2/x+2*y^3/x^2)-y):
The flow psi.
> simplify(diff(a,x)*(alpha-x)+diff(a,y)*(beta-y)+a): 
  simplify(diff(b,x)*(alpha-x)+diff(b,y)*(beta-y)+b):
Verification that both coordinates of psi satisfy the linear PDE.
> pi:=2*x^2-3*x*y: rho:=x*y-2*y^2: 
Vector field of the flow phi.
> u:=(x*y*sqrt(1-2*x+2*y)+x^2)/((y+x+2*y^2)*(1-2*x+2*y)):
  v:=(y^2*sqrt(1-2*x+2*y)+y*x)/((y+x+2*y^2)*sqrt(1-2*x+2*y)):
The flow phi.
> simplify(diff(u,x)*(pi-x)+diff(u,y)*(rho-y)+u):
  simplify(diff(v,x)*(pi-x)+diff(v,y)*(rho-y)+v):
Verification that both coordinates of phi satisfy the linear PDE. 
Initial conditions are easily verified by hand.
> simplify(diff(pi,x)*alpha+diff(pi,y)*beta-pi*diff(alpha,x)-rho*diff(alpha,y)): 
  simplify(diff(rho,x)*alpha+diff(rho,y)*beta-pi*diff(beta,x)-rho*diff(beta,y)):
Verification that flows comumute; that is, that Lie bracket vanish.
> A:=unapply(a,(x,y)):B:=unapply(b,(x,y)):
  U:=unapply(u,(x,y)):V:=unapply(v,(x,y)):
> simplify(A(U(x,y),V(x,y))-U(A(x,y),B(x,y))): 
  simplify(B(U(x,y),V(x,y))-V(A(x,y),B(x,y))):
Verification that flows commute.

> restart: Proposition 7
> alpha:=-x*y^3/(x-y)^2:beta:=(3*x*y^3-2*y^4)/(x-y)^2: 
Vector field of the flow psi.
> F:=(9*a*x-8*x^2-3*a*y)*x^2*y^2+(3*y+6*y^2-x)*a*(a*y-3*a*x+3*x^2)^2: 
Equation for the function "a".
> factor(F,I):
T:=-diff(F,x)/diff(F,a)*(alpha-x)-diff(F,y)/diff(F,a)*(beta-y)+a:
Verification that the first coordinate of psi satisify the linear PDE.
> factor(numer(T)/F):
Verification that the PDE holds modulo F.
> pi:=2*x^2+x*y:rho:=x*y+2*y^2:
> simplify(diff(pi,x)*alpha+diff(pi,y)*beta-pi*diff(alpha,x)-rho*diff(alpha,y)):
> simplify(diff(rho,x)*alpha+diff(rho,y)*beta-pi*diff(beta,x)-rho*diff(beta,y)):
Verification that flows comumute; that is, that Lie brackets vanish.
> r:=((y-3*x+6*x^2)/(x-3*y+6*y^2))^(1/2):
> u:=((x+y*r)^(1/3)+(x-y*r)^(1/3))*x*(x-y)
/(x-3*y+6*y^2)^(1/3)/(y-3*x+6*x^2)-2*x^2/(y-3*x+6*x^2):
> simplify(diff(u,x)*(pi-x)+diff(u,y)*(rho-y)+u):
Verification that the function u satisfies the PDE.
\end{verbatim}

\bibliographystyle{amsplain}

\end{document}